%% file: neurips_2019.tex
\documentclass{article}

    \usepackage[preprint,nonatbib]{neurips_2019}

\usepackage[utf8]{inputenc} 
\usepackage[T1]{fontenc}    

\usepackage{xcolor}
\definecolor{HalfGray}{gray}{0.55}
\definecolor{OliveGreen}{rgb}{0,.35,0}
\definecolor{webbrown}{rgb}{.6,0,0}
\definecolor{BrightViolet}{rgb}{0.5,0.2,0.8}
\definecolor{Maroon}{cmyk}{0, 0.87, 0.68, 0.32}
\definecolor{RoyalBlue}{cmyk}{1, 0.50, 0, 0.25}
\definecolor{Black}{cmyk}{0, 0, 0, 0}
\usepackage[unicode]{hyperref}
\hypersetup{
colorlinks=true,
linktocpage=true,
pdfstartview=FitH,
breaklinks=true,
pdfpagemode=UseNone,
pageanchor=true,
pdfpagemode=UseOutlines,
plainpages=false,
bookmarksnumbered,
bookmarksopen=false,
bookmarksopenlevel=1,
hypertexnames=true,
pdfhighlight=/O,
urlcolor=Maroon, linkcolor=RoyalBlue, citecolor=OliveGreen, 
pdftitle={},
pdfauthor={},
pdfsubject={},
pdfkeywords={},
pdfcreator={pdfLaTeX},}

\usepackage{url}            
\usepackage{booktabs}       
\usepackage{amsfonts}       
\usepackage{nicefrac}       
\usepackage{microtype}      
\usepackage{amsthm,amsmath,amssymb,enumerate,bbm,tikz,stmaryrd,subcaption}
\newtheorem{theorem}{Theorem}[section]
\newtheorem{Assumption}{Assumption}
\newtheorem{proposition}[theorem]{Proposition}
\newtheorem{corollary}[theorem]{Corollary}
\newtheorem{lemma}[theorem]{Lemma} \theoremstyle{definition}

 \theoremstyle{remark}
\newtheorem{remark}[theorem]{Remark}

\input{commandes_guillaume}

\renewcommand{\leq}{\leqslant}
\renewcommand{\geq}{\geqslant}

\DeclareMathOperator*{\med}{med}

\DeclareMathOperator*{\argmin}{arg\,min}

\newcommand{\bayes}{f^*}

\title{A MOM-based ensemble method for robustness, subsampling and hyperparameter tuning}

\author{Joon Kwon\\INRA \& AgroParisTech\\Université
  Paris--Saclay\And Guillaume
  Lecué\\CREST\\Université
  Paris--Saclay\And Matthieu
  Lerasle\\CNRS \& Université Paris--Sud\\École polytechnique}

\begin{document}

\maketitle

\begin{abstract}
Hyperparameters tuning and model selection are important steps in
  machine learning. Unfortunately, classical hyperparameter
  calibration and model selection procedures are sensitive to outliers
  and heavy-tailed data.  In this work, we construct a selection
  procedure which can be seen as a robust alternative to
  cross-validation and is based on a median-of-means principle \cite{MR1688610,
    MR855970, MR702836}. Using this procedure, we also build
  an ensemble method which, trained with algorithms and corrupted
  heavy-tailed data, selects an algorithm, trains it with a large
  uncorrupted subsample and automatically tune its
  hyperparameters. The construction relies on a
  divide-and-conquer methodology \cite{Jor:2013}, making this method
  easily scalable for autoML given a corrupted database.  This method
  is tested with the LASSO \cite{MR1379242} which is known to be
  highly sensitive to outliers.
\end{abstract}

\section{Introduction}
Robustness has become an important subject of interest in the machine
learning community over the last few years because large datasets are
very likely to be corrupted.  This may happen due to hardware, storage
or transmission issues, for instance, or as a result of (human)
reporting errors.  As can be seen, for instance, in Figure~1 in
\cite{lecue2017robust} and Figure~1 and 5 in \cite{lecue2018robust},
many learning algorithms based on emprical risk minimization (including
the LASSO) may be completely mislead by a single corrupted example.



Robust alternatives to empirical risk minimizers and their
penalized/regularized versions have been studied in density estimation
\cite{BaraudBirgeSart} and least-squares regression
\cite{MR2906886,LugosiMendelson2016,Fan1, Fan2, Fan3}.  Various robust
descent algorithms have also been recently considered
\cite{Ravikumar2018, oliveira2017sample, oliveira2017sample_2,
  holland2018robust,holland2018classification}.  Despite these
important progresses, the final steps of a data-scientist routine,
which are estimator selection and hyperparameter tuning
\cite{MR1679028, MR1462939, MR3224300} are yet to receive a proper
treatment. In fact, practitioners usually have at disposal several
algorithms, each of these requiring one or several parameters to be
tuned.  An alternative to estimator selection is aggregation
(aka ensemble methods) \cite{MR2163920, MR2051002,MR1775638} which
outputs e.g.\ a linear or convex combination of the candidate estimators;
classical examples include binning, boosting, bagging or stacking.

The most common procedure used to select or aggregate candidate
estimators is (cross-)validation: the dataset is partitioned (several
times in the case of cross-validation) into a \emph{training sample} used to
build candidate estimators and a \emph{test sample} used to estimate
their risks.  The final estimator is either
the candidate with lowest estimated risk, or a linear combination of
the candidates with coefficients depending on the estimated risks.
Even if some candidate estimators are robust, outliers from the test set may
mislead the selection/aggregation step, resulting in a
poor final estimator. This raises the question of a robust
selection/aggregation procedure, which is addressed in the present work.

There exist many data-driven methods to tune hyperparameters or to
select an estimator from a collection of candidates. Among these, one
can mention the SURE method \cite{stein1981estimation}, model
selection
\cite{MR1679028,MR1462939,MR1848946,MR2242356,MR3766945,MR2291502}
where penalization methods are used to select among candidates built
with \emph{the same data} as those used to build the original
estimators, selection, convex or linear aggregation
\cite{MR2051002,MR2356821,MR2198228}, cross-validation
\cite{MR2602303,MR3595142} or Lepski \cite{MR1147167} and the
Goldenschluger-Lepski \cite{goldenshluger:lepski:2008} methods to name
a few. To the best of our knowledge, all these techniques either use a
classical non-robust validation principle or estimate the risk with
the non-robust empirical risk. A notable exception is the estimator
selection procedure of \cite{MR2834722,MR3224300} which is robust in
general settings \cite{MR2834722} and extremely efficient in Gaussian
linear regression \cite{MR3224300}. The main drawback is that this
procedure requires robust tests in Hellinger distance that may be hard
to compute for general learning problems where one does not specify
statistical models with bounded complexities.

The first contribution of this paper is a general and robust estimator
selection procedure with provable theoretical guarantees, which can
be viewed as a robust alternative to cross-validation.  Roughly
stated, the procedure uses a median-of-means principle
\cite{MR1688610, MR855970, MR702836} to build robust pairwise
comparisons between candidates, and the final estimator is then
selected by a minmax procedure in the spirit of \cite{MR2906886,MOM2}
or the Goldenshluger-Lepski method \cite{goldenshluger:lepski:2008},
see Section~\ref{sec:estimator-selection} for details.  The method is
easily implementable. We here focus on least-squares regression and refer
to \cite{LerOli2011} for other examples including density estimation
and classification.

The second contribution is the definition of an ensemble method based on this
selection procedure and a subsampling strategy.  Two of the main ideas
behind this method is that subsampling can provide robustness by
avoiding outliers and that the choice of the subsample can itself be
seen an a hyperparameter to be tuned.  Estimator selection
procedures can then be used to simultaneously select the best algorithm,
an uncorrupted subsample and the best hyperparameters.
Moreover, the method is computationally attractive.

Subsampling is usually used in machine learning for
computational reasons: some algorithms require to break large datasets
into smaller pieces \cite{Jor:2013}, for instance
in supervised learning \cite[for classification and
regression]{Che_Xie:2014} and \cite[for matrix
factorization]{Mac_Taw_Jor:2015}.
A natural way to divide-and-conquer corresponds to the older idea of
\emph{subagging} \cite[subsample aggregating]{Buh:2003}---which is
a variant of bagging \cite{Bre:1996a}:
one randomly chooses several small subsets of data, 
build an estimator from each subsample, 
and aggregate them into a single estimator. 
For instance, the bag of little bootstraps \cite{Kle_etal:2014} 
builds confidence intervals in such a way. 
Subagging is also used for large-scale sparse regression \cite{Bra:2016}. 

The paper is divided as follows. Section~\ref{sec:setting} presents
the general prediction setting we consider, and
Section~\ref{sec:estimator-selection} introduces the robust estimator
selection procedure. Theoretical guarantees for the latter are given in
Theorem~\ref{thm:oracle-inequality}.  The ensemble method is defined
in Section~\ref{sec:Subsampling} and applied to the LASSO in
Section~\ref{sec:applications}.
Applications to the ERM in linear aggregation are presented in Appendix~\ref{sec:appl-erm-line}.
Numerical experiments are presented in
Section~\ref{sec:numer-exper}.  The proofs are outsourced in the
appendix in Appendices~\ref{sec:Proof} and~\ref{Proofs:Main}.

\section{Setting}\label{sec:setting}
For positive integers $k\leqslant l$, let $[k]=\left\{
  1,2,\dots,k \right\}$, $\left\llbracket k,l \right\rrbracket
=\left\{ k,k+1,\dots,l \right\}$, and
reversed double-bar brackets mean exclusion of the corresponding
integer, e.g.\ $\left\rrbracket k,l \right\rrbracket =\left\{
  k+1,k+2,\dots,l \right\}$. We call partition of a set $E$ any family of disjoint subsets of $E$ with union equal to $E$.

Let $\mathbb{X}$ be a measurable space. Let $P$ be a probability
distribution on $\mathbb{X}\times \mathbb{R}$, and let $(X,Y)\sim P$. Denote $P_X$
the marginal distribution of $X$. Assume that $\mathbb{E}\left[ Y^2
\right]<+\infty$. Denote $L^2(P_X)$ the Hilbert space of measurable functions $f:\mathbb{X}\to \mathbb{R}$ such that $\mathbb{E}\left[ f(X)^2 \right]<+\infty$, the norm being denoted by $\left\| f \right\|  =\sqrt{\mathbb{E}\left[ f(X)^2 \right] }$.
For any probability measure $Q$ on $\mathbb{X}\times \mathbb{R}$ and any measurable function $g:\mathbb{X}\times \mathbb{R}\to
\mathbb{R}$, that belongs to $L^1(Q)$, let $Q\left[ g \right]:=\mathbb{E}_{Z\sim Q}\left[ g(Z) \right]$. 


Let $F$ be a linear subspace of $L^2(P_X)$. We call \emph{estimator}
any element of $F$. For $f\in F$, let $\gamma(f):\mathbb{X}\times \mathbb{R}\to \mathbb{R}$ denote the square-loss function associated with $f$, defined by for all $(x,y)\in \mathbb{X}\times \mathbb{R}$ by $\gamma(f)(x,y)=(y-f(x))^2$. For any function $f:\mathbb{X}\to \mathbb{R}$ in $L^2(P_X)$, let $R(f)$ denote its \emph{risk} $R(f):=P\left[ \gamma(f) \right]$ and let $\bayes$ be the \emph{oracle}:
$f^*:=\argmin_{f\in F}R(f)$.
Let $\ell$ denote the \emph{excess risk} with respect to
$f^{*}$:
\[ \ell(f)=R(f)-R(f^*)=P\left[ (f-f^*)^2 \right]=\left\| f-f^*
  \right\|^2. \] The second equality holds since $F$ is a linear space. A \emph{learning algorithm} is a measurable map $G\colon
\bigcup_{n=1}^{+\infty}(\mathbb{X}\times \mathbb{R})^n\to F$ which
takes a dataset of any (finite) size as input and outputs an estimator in $F$.
\begin{Assumption}\label{Ass:L4L2}Let $\chi,\sigma>0$ such that for every $f\in F$,
  \[(P f^4)^{1/4}\leq \chi (Pf^2)^{1/2} \quad \text{and}\quad P\left[ (Y-f^*)^2(f-f^*)^2 \right]\leq \sigma^2 P(f-f^*)^2. \]
\end{Assumption}
This assumption only involves second and fourth moments. The first
assumption $(P f^4)^{1/4}\leq \chi (Pf^2)^{1/2}$ is satisfied for instance by linear functions $f(\cdot)=\inr{\cdot, t}$ for $t\in\bR^d$ and $X$ which is a $d$-dimensional vectors with independent entries with a fourth moment \cite{Shahar-COLT}. It therefore covers heavy-tailed cases beyond classical $L_\infty$-boundedness or subgaussian assumptions. The second assumption $P\left[ (Y-f^*)^2(f-f^*)^2 \right]\leq \sigma^2 P(f-f^*)^2$ holds for instance when the noise $Y-f^*(X)$ is independent of $X$ and has a second moment---which is a very standard statistical modeling assumption when $Y=f^*(X)+\zeta$ with $\zeta$ independent of $X$. It also also holds when $Y-f^*(X)$ has a fourth moment by using Cauchy-Schwarz.

Let $N\geqslant 1$ be the size of the dataset $(X_i,Y_i)_{i\in [N]}$,
which is partitioned into informative data and outliers:
$[N]= \mathcal{O}\sqcup \mathcal{I} $. Informative data
 $(X_i,Y_i)_{i\in \mathcal{I}}$ is assumed independent and
 identically distributed (i.i.d.), with common distribution $P$. No
 assumption is granted on outliers $(X_i,Y_i)_{i\in \mathcal{O}}$. Of
 course, the partition $\mathcal{O}\sqcup \mathcal{I}$ is unknown to
 the learner.
 We call a \emph{subsample} any nonempty subset  $B\subset [N]$ (or the
 corresponding data $(X_i,Y_i)_{i\in B}$),
 and for any measurable function $g:\mathbb{X}\times \mathbb{R}\to
 \mathbb{R}$, denote:
 \[ P_B\left[ g \right]=\frac{1}{\left| B \right| }\sum_{i\in B}^{}g(X_i,Y_i). \]
 
 \section{Minmax-MOM selection: a robust alternative to
   cross-validation}
\label{sec:estimator-selection}
Let $(\hat{f}_m)_{m\in \mathcal{M}}$ be a finite collection of
estimators.  For each index $m\in \mathcal{M}$, we assume that
there exists a learning algorithm $G_m$ and a subsample $B_m\subset
[N]$ of cardinality less than $N/4$ such that $\hat{f}_m$ is the
estimator trained by algorithm $G_m$ using subsample $B_m$; in other
words $\hat{f}_m=G_{m}((X_i,Y_i)_{i\in B_m})$ (the
remaining of the dataset will be used to estimate the risk of the
estimator, like in cross-validation).
The best choice of $m\in\cM$ regarding our final objective satisfies
\begin{align}\label{eq:EstByTests}
  \begin{split}
    m_o&:=\argmin_{m\in \cM}P\left[ \gamma(\hat{f}_m) \right]=\argmin_{m\in \cM}\max_{m'\in\cM}P\left[  \gamma(\hat{f}_m)-\gamma(\hat{f}_{m'}) \right]\enspace. 
  \end{split}
\end{align}
However, the real-valued expectations
$P\big[\gamma(\hat{f}_m)-\gamma(\hat{f}_{m'}) \big]$ are unknown: let
us construct a robust estimator of those quantities.
Let $V\in \left\llbracket 1,N/8 \right\rrbracket$.
For each couple $(m,m')\in \mathcal{M}^2$, let $(T_v^{(m,m')})_{v\in
  [V]}$ be a partition into $V$ blocks of a subset of $[N]\setminus (B_m\cup B_{m'})$, such that $\big| T_v^{(m,m')} \big| \geqslant N/4V$ for all $v\in [V]$. The estimates $\mathcal{T}(m,m')$ of $P\big[ \gamma(\hat{f}_m)-\gamma(\hat{f}_{m'}) \big]$ are defined by:
\[ \mathcal{T}(m,m'):=\med_{v\in [V]}\left\{  P_{T_v^{(m,m')}}\left[ \gamma(\hat{f}_m)-\gamma(\hat{f}_{m'}) \right] \right\},  \]
in other words, $\mathcal{T}(m,m')$ is the median of the $V$ empirical
means $P_{T_v^{(m,m')}}\big[
\gamma(\hat{f}_m)-\gamma(\hat{f}_{m'})\big], v\in[V]$. The selection
of the final estimator is obtained by plugging these median-of-means (MOM)
estimators into equation \eqref{eq:EstByTests}. In other words, we select $\hat{f}_{\widehat{m}}$, where
\begin{equation}\label{def:MinmaxMOMSelect}
 \hat{m}:=\argmin_{m\in \mathcal{M}}\max_{m'\in \mathcal{M}}\mathcal{T}(m,m').  
\end{equation}Thanks to the median-of-mean operator, the risk of the selected
estimator $\hat{f}_{\hat{m}}$ is expected to be close to the risk of
the best estimator  $\hat{f}_{m_o}$, even for heavy-tailed and corrupted data because $\mathcal{T}(m,m')$ is a robust (to outliers) sub-gaussian estimator of $\gamma(\hat{f}_m)-\gamma(\hat{f}_{m'})$,
even for heavy-tailed data \cite{LerOli2011}.

Median-of-means have been introduced in \cite{MR1688610, MR855970, MR702836}. 
Median-of-means pairwise comparisons have been used to build robust estimators in \cite{LugosiMendelson2016,MOM1}.
Minmax strategies have been used in \cite{MR2906886,MOM2} for least-squares regression and in \cite{BaraudBirgeSart} for density estimation. 
Finally, the minmax principle has been used for (non-robust) selection of estimators in \cite{goldenshluger:lepski:2008}.
\begin{remark}[Minmax-MOM selection to divide-and-conquer]
\label{sub:the_minmax_mom_selection_procedure_as_a_divide_and_conquer_methodology}
It is classical to use divide-and-conquer approaches \cite{Jor:2013} to deal with large databases: the database is divided in small batches, algorithms are run on each batch and the results are ``aggregated''. 
Minmax-MOM selection procedure \eqref{def:MinmaxMOMSelect} can perform
this king of aggregation.  Denote by $B_{m}$ the block of data hosted
on server $m\in \mathcal{M}$. Train estimators $\hat f_m$ for all
$m\in\cM$. Then, for all $m,m'\in\cM$, compute the $V$ real numbers
$P_{T_v^{(m,m')}}\big[ \gamma(\hat{f}_m)-\gamma(\hat{f}_{m'}) \big],\
v\in [V]$ and take their median. Then compute the minmax-MOM estimator
(if there are too many medians, choose $m$ and $m'$ at random in $\mathcal{M}$).  Following the map-reduce terminology \cite{dean2008mapreduce}, the mapper is the training of the procedure itself and the $V$ evaluations.  The reducer is the computation of the $\binom{\left| \mathcal{M} \right| }{2}$ medians of differences of empirical risks and the minmax-MOM selection \eqref{def:MinmaxMOMSelect}.
\end{remark}
\begin{theorem}[Robust oracle inequality]
\label{thm:oracle-inequality}
Grant Assumption~\ref{Ass:L4L2} and assume $V\in \left\llbracket 3\left| \mathcal{O} \right| ,N/8 \right\rrbracket$. Then with
probability larger than $1-\left| \mathcal{M} \right|^2e^{-V/48}$, the
estimator $\hat{f}_{\widehat{m}}$, where $\hat{m}$ is selected by the
minmax-MOM selection procedure \eqref{def:MinmaxMOMSelect} satisfies, for all
$\varepsilon>0$,
\[ (1-a_{\varepsilon,V})\, \ell(\hat{f}_{\widehat{m}})\leqslant (1+3a_{\varepsilon,V})\, \min_{m\in \mathcal{M}}\ell(\hat{f}_{m})+2b_{\varepsilon,V}, \]
where $f\mapsto \ell(f) = R(f) - R(f^*)$ is the excess loss function, $a_{\varepsilon,V}:=8\chi^2\sqrt{2V/N} + 2\sqrt{2}\varepsilon$ and $b_{\varepsilon,V}:=(64V\sigma^2)/N \varepsilon$.
\end{theorem}

The proof of Theorem~\ref{thm:oracle-inequality} is postponed to
Section~\ref{sec:ProofMainThm}. Roughly speaking,
Theorem~\ref{thm:oracle-inequality} states that, with exponentially
large probability, the selected estimator \eqref{def:MinmaxMOMSelect}
has the excess risk of the best estimator in the collection
$(\hat{f}_m)_{m\in \mathcal{M}}$.  Following \cite{MR1311089}, this
result is called an oracle inequality.  We call it \emph{robust} as it
holds under moment assumptions on the linear space $F$ (see
Assumption~\ref{Ass:L4L2}) and for a dataset that may contain
outliers.  The residual term $b_{\varepsilon, V}$ is of order $V/N$.
If $\log |\cM|\gtrsim |\cO|$ and $V\asymp \log|\cM|$, the residual
term is of order $\log |\cM|/N$, which is minimax optimal
\cite{TsyCOLT07}.  The oracle inequality is interesting when
$a_{\varepsilon, V}<1$ which holds if $\chi\lesssim \sqrt{N/V}$.  The
``constant'' $\chi$ in Assumption~\ref{Ass:L4L2} may therefore grow
with the dimension of $F$ as in the examples of
\cite{saumard2018optimality} without breaking the results.


\section{An ensemble method to induce robustness, subsampling and hyperparameters tuning}\label{sec:Subsampling}
In this section, we define an \emph{ensemble method} which takes one or
several (non necessarily robust) algorithms as input and ouputs an estimator. The method
is robust to the presence of outliers, has subsampling capabilities,
and automatically tune hyperparameters.  The main ideas
behind the construction are: using subsampling as a way of
achieving robustness (by avoiding outliers), viewing the choice of the
subsample as a hyperparameter to be tuned, and using the robust
selection procedure from Section~\ref{sec:estimator-selection} to
select the final estimator.  Performance guaranties are established in Corollary~\ref{Cor:AppliSelSousEns}.

\subsection{Definition of the method}
\label{sec:defin-ensemble-meth}

Let $(G_{\lambda})_{\lambda\in \Lambda}$ be a finite collection of
learning algorithms which outputs estimators in $F$. The collection
may in fact correspond to a single algorithm with several combinations
of hyperparameters values, or even several different algorithms with
several combinations of hyperparameters values.

We now construct the set $\mathcal{B}$ of subsamples to be considered
by the method.  Assume
$N\geqslant 8$.  Let
$K_{\text{min}}$ and $K_{\text{max}}$ integers such that
$3\leqslant K_{\text{min}}\leqslant K_{\text{max}}\leqslant \log_2N$
(these parameters will specify the subsamples cardinality range).  For
each
$K\in \left\llbracket K_{\text{min}},\ K_{\text{max}}\right\rrbracket
$, consider a partition $(B_k^{(K)})_{k\in [2^K]}$ of $[N]$ such that
for all $k\in [2^K]$,
$\left\lfloor N / 2^K \right\rfloor \leqslant \left| B_k^{(K)}
\right|$. We call $(B_k^{(K)})_{k\in [2^K]}$ \emph{the $2^K$-partition}.
Let
\begin{equation}\label{eq:model_class}
\mathcal{B}=\bigcup_{K=K_{\text{min}}}^{K_{\text{max}}}\bigcup_{k\in [2^K]}\left\{
  B_k^{(K)}\right\}\quad \text{and}\quad  \cM = \Lambda\times \mathcal{B} .
\end{equation} From $K_{\text{min}}\geqslant 3$ we can easily deduce
that each subsample in $\cB$ has cardinality less than $N/4$.  
Then, as in Section~\ref{sec:estimator-selection}, for each
$m=(\lambda,B)$, let $\hat{f}_m$ be the estimator trained by
algorithm $G_{\lambda}$ using subsample $B$, in other words: $\hat{f}_m=G_{\lambda}((X_i,Y_i)_{i\in B})$.
Let $V\in \left\llbracket 3,N/8 \right\rrbracket$. For each couple
$(m,m')\in \mathcal{M}^2$, let $(T_v^{(m,m')})_{v\in [V]}$ be a
partition of a subset of $[N]\setminus(B_m\cup B_{m'})$ such that
$\left| T_v^{(m,m')} \right| \geqslant N/4V$ for all $v\in [V]$.
Then, using the minmax-MOM selection procedure \eqref{def:MinmaxMOMSelect} from
Section~\ref{sec:estimator-selection}, we select from collection
$(\hat{f}_m)_{m\in \cM}$ the final estimator $\hat{f}_{\hat{m}}$.

\subsection{Theoretical guarantees}
The following result states that if risk bounds hold for
algorithms $(G_{\lambda})_{\lambda\in \Lambda}$ in a context with
no-outlier, then $\hat{f}_{\widehat{m}}$ essentially satisfies the
best of those risk bounds, even in the presence of outliers.
\begin{corollary}\label{Cor:AppliSelSousEns} Let $\cM$ be defined by \eqref{eq:model_class}. Grant Assumption~\ref{Ass:L4L2}. 
  Let $\rho:\Lambda\times\mathbb{N}^*\to \mathbb{R}_+\cup \left\{ +\infty \right\}$ be a
  non-increasing function in its second variable and
  $\nu:\Lambda\to \mathbb{R}_+^*$.  Denote
  $\nu_{\text{max}}:=\lceil \max_{\lambda\in \Lambda}\nu(\lambda)
  \rceil$.  Assume that $N\geqslant \nu_{\text{max}} \max_{}(8V,2^{K_{\text{min}}+1})$ and
  $V\in \left\llbracket 3\left| \mathcal{O} \right|, 2^{K_{\text{max}}-1}\right\rrbracket$.  Assume that, for all
  $\lambda\in \Lambda$ and $B\subset \mathcal{I}$ such that
  $\left| B \right| \geqslant \nu(\lambda)$, it holds that
  $\ell(\hat{f}_{\lambda,B})\leqslant \rho(\lambda,\left| B \right|)$
  with probability larger than $1-\exp(-1/48)$. Then for all $\varepsilon>0$, the estimator
  $\hat{f}_{\widehat{m}}$ defined in \eqref{def:MinmaxMOMSelect}
  satisfies
 \begin{equation}\label{eq:residue_rate_coro}
 (1-a_{\varepsilon,V})\ell(\hat{f}_{\widehat{m}})\leqslant
  (1+3a_{\varepsilon,V})\, \min_{\lambda\in \Lambda}\rho\left(   \lambda,\left\lfloor  \frac{N}{\max_{}(4V,2^{K_{\text{min}}})} \right\rfloor\right)+2b_{\varepsilon,V}
 \end{equation}
with probability larger than
\begin{equation}\label{eq:proba_dev_coro}
1-(\left| \Lambda \right|^2N^2+1)e^{-V/48}.
\end{equation}
\end{corollary}
Corollary~\ref{Cor:AppliSelSousEns} is proved in
Section~\ref{sec:ProofCorSubSel}. 
Let us stress some important aspects.

Estimators $\hat{f}_{\lambda,B}$ for $(\lambda, B)\in\cM$ are assumed
to satisfy an excess risk bound with rates
$\rho(\lambda,\left| B \right|)$ only with constant probability (the
constant $1-\exp(-1/48)$ chosen here has nothing special), when $B$ is
large enough and only contains informative data.  For example, this
condition is met by ERM when informative data satisfies moment
assumptions, see Propositions~\ref{coro:minmax_MOM_LASSO} and
\ref{prop:GuillaumeShahar}.  With these arguably weak requirement, the
above statement claims that the ensemble method achieves the best bound
among
$\rho(\lambda,\left\lfloor N/\max_{}(4V,2^{K_{\text{min}}})
\right\rfloor ),\ \lambda\in \Lambda$ with exponentially large
probability \eqref{eq:proba_dev_coro}.

        The upper bound $\rho(\lambda,\left| B \right| )$ on the
        excess risk of $\hat{f}_{\lambda,B}$ depends on $\lambda$ and
        the size $\left| B \right| $ of the subsample. It improves
        with the sample size by the monotonicity assumption on $\rho$.

Finally, the function $\lambda\mapsto \nu(\lambda)$ is introduced to handle situations where the risk bound holds only when the sample size is larger than $\nu(\lambda)$.



 \subsection{An efficient partition scheme of the dataset}
 \label{sec:an-effic-constr}
 The partitions $(B_k^{(K)})_{k\in \left[ 2^K \right]}$
 ($K\in \left\llbracket K_{\text{min}},\ K_{\text{max}}
 \right\rrbracket$) and $(T_v^{(m,m')})_{v\in [V]}$ (for
 $(m,m')\in \mathcal{M}^2$) can be constructed in many different ways.
 This section presents a specific choice for those partitions
  which yields a computational advantage by
 significantly reducing the number of empirical risks
 $P_{T_v^{(m,m')}}[\gamma(\hat{f}_m)]$ to be computed (by making many
 of them redundant). This complexity reduction makes the computations
 from Section~\ref{sec:numer-exper} possible in a reasonable amount of time.

 The minimax-MOM selection procedure \eqref{def:MinmaxMOMSelect}
 requires, for all $(m,m')\in \mathcal{M}^2$ and $v\in [V]$, the
 computation of $P_{T_v^{(m,m')}}\big[ \gamma(\hat{f}_m) \big]$ and
 $P_{T_v^{(m,m')}}\big[ \gamma(\hat{f}_{m'}) \big]$.  Since the
 partition $(T_v^{(m,m')})_{v\in [V]}$ may be different for each couple
 $(m,m')\in \mathcal{M}^2$, this requires, in the worst case, the
 computation of $V\left| \mathcal{M} \right|^2$ empirical risks.
By comparison, the
construction presented here will only require the computation of $8V|\cM|/3$ empirical risks.

For $K\in \left\llbracket 3,\ \lfloor\log_2N\rfloor
\right\rrbracket$ and $k\in [2^K]$, define $B_k^{(K)}:=\left\rrbracket \left\lfloor \frac{(k-1)N}{2^K} \right\rfloor
    ,\ \left\lfloor \frac{kN}{2^K} \right\rfloor \right\rrbracket$.
For each $K\in \left\llbracket 3,\ \lfloor\log_2N\rfloor \right\rrbracket$, $(B_k^{(K)})_{k\in [2^K]}$ is a partition of $[N]$ such that, for
each $k\in [2^K]$, $\lfloor N/2^K\rfloor\leqslant \big| B_k^{(K)}
\big| \leqslant N/4$, as required. Moreover,
the following key property holds.
\begin{lemma}\label{lem:key1}
 Let $3\leqslant K'\leqslant K\leqslant \lfloor\log_2N\rfloor$.

\begin{enumerate}[(i)]
\item\label{item:2} For all $k\in [2^K]$, $B_k^{(K)}\subset B^{(K')}_{\lfloor(k-1)2^{K'-K}\rfloor +1}$.
  
\item\label{item:3} For all $k'\in [2^{K'}]$, $(B_k^{(K)})_{k\in \left\rrbracket       (k'-1)2^{K-K'},\ k'2^{K-K'} \right\rrbracket}$ is a partition
  of $B_{k'}^{(K')}$.
\end{enumerate}
\end{lemma}


Let
\begin{equation}
\label{eq:2}
K_0:=\lceil \log_2(V/3)\rceil+2.
\end{equation}
For all $K_1,K_2\in \left\llbracket 3,\ \lfloor\log_2N\rfloor \right\rrbracket$ and $k_1\in \left[ 2^{K_1} \right]$, $k_2\in
\left[ 2^{K_2} \right]$, let $\mathcal{K}_0(K_1,k_2,K_2,k_2)$
be the set of indices from the $2^{K_0}$-partition which have
empty intersection with both $B_{k_1}^{(K_1)}$ and $B_{k_1}^{(K_2)}$:
\[   \mathcal{K}_0(K_1,k_1,K_2,k_2)
:=\left\{ k\in \left[ 2^{K_0}
    \right]\,\middle|\,B_k^{(K_0)}\cap (B_{k_1}^{(K_1)}\cup B_{k_2}^{(K_2)})=\varnothing \right\}. \]
\begin{lemma}
  \label{lm:1}
  For all $3\leqslant K_1,K_2\leqslant \lfloor\log_2N\rfloor$ and
  $k_1\in \left[ 2^{K_1} \right]$, $k_2\in \left[ 2^{K_2} \right]$, we have $\left| \mathcal{K}_0(K_1,k_1,K_2,k_2) \right| \geqslant V$.
\end{lemma}
Let $(m,m')\in \mathcal{M}^2$ and $K_1,k_1,K_2,k_2$ be such that $B_m=B_{k_1}^{(K_1)}\quad \text{and}\quad B_{m'}=B_{k_2}^{(K_2)}$.  Then, the collection of sets $(B_k^{(K_0)})_{k\in \mathcal{K}_0(K_1,k_1,K_2,k_2)}$ is a sub-collection of the $2^{K_0}$-partition,
whose sets have empty intersection with both $B_m$ and $B_{m'}$, and which, according to
Lemma~\ref{lm:1}, contains at least $V$ sets.
We can thus define $(T_v^{(m,m')})_{v\in [V]}$ as a sub-collection of size exactly $V$.
Consequently, $(T_v^{(m,m')})_{v\in [V]}$ is indeed a partition of a subset of
$[N]\setminus (B_m\cup B_{m'})$. Moreover, we have the following lower
bound on the cardinality of its sets, which is required (see Section~\ref{sec:estimator-selection}).
\begin{lemma}\label{lem:key3}
  For all $(m,m')\in \mathcal{M}^2$ and $v\in [V]$, $\left| T_v^{(m,m')} \right|\geqslant  N/(4V)$. 
\end{lemma}
Consequentely, to compute the minimax-MOM selection procedure in the
context of the ensemble method defined in
Section~\ref{sec:defin-ensemble-meth}, the empirical risk of each
estimator $\hat{f}_m$ has to be computed on the $2^{K_0}$-partition
only, which thanks to \eqref{eq:2} means the computation of at most
$8V|\cM|/3$ empirical risks, as advertised.

\section{Application to fine-tuning the regularization parameter of the LASSO}
This section applies the ensemble method from
Section~\ref{sec:Subsampling} with the LASSO as input algorithm.
\label{sec:applications}
Consider here $\mathbb{X}=\mathbb{R}^d$, and denote
$\hat \beta_{\lambda, B}$ the LASSO estimator trained with
regularization parameter $\lambda$ and subsample $B$.
\begin{equation}
\label{eq:lasso}
\hat{\beta}_{\lambda,B}=\argmin_{\beta\in \mathbb{R}^d}\left\{ \frac{1}{\left| B \right| }\sum_{i\in B}^{}(Y_i-\left< \beta ,X_i \right> )^2+\lambda\left\| \beta \right\|_1 \right\}. 
\end{equation}
 Statistical guarantees for the LASSO, which we recall below, have been obtained in Theorem~1.4 in \cite{LMregcomp} with
a regularization parameter
$\lambda\asymp\norm{\zeta}_{L_q}\sqrt{s \log(ed/s)/N}$ instead of
$\norm{\zeta}_{L_q}\sqrt{s \log(ed)/N}$. This choice is valid under
the following assumption.

\begin{Assumption}\label{ass:lasso}
Let $(X,Y)\sim P$. For all $t\in\R^d, \E \inr{X, t}^2=\norm{t}_2^2$ and there exists $L>0$ such that, for all $p\geq1$ and $t\in\R^d, (\E|\inr{X,t}|^p)^{1/p}\leq L \norm{t}_2$. Moreover, there exists $q>2$ such that, for any $\beta^*\in\argmin_{\beta\in\R^d}\E(Y-\inr{X, \beta})^2$, $\zeta:=Y-\inr{X, \beta^*}\in L_q$.  
\end{Assumption}

\begin{proposition}\label{prop:lasso_basis_estimators}
Grant Assumption~\ref{ass:lasso}.  Assume that $\beta^*$ is $s_0$-sparse for some $s_0\in [d]$.  Let $B\subset \cI$ be such that $|B|\geq s_0 \log(ed/s_0)$. Then, there exist absolute constants $c_0$ and $c_1$ such that the LASSO with regularization parameter $\lambda = c_0\norm{\zeta}_{L_q}\sqrt{s_0 \log(ed/s_0)\left| B \right|^{-1}}$ satisfies, with probability at least $1-\exp(-1/48)$, 
\begin{equation}\label{eq:oracle_ineq_lasso}
\ell(\hat \beta_{\lambda, B})\leq c_1 \norm{\zeta}_{L_q}^2\frac{s_0 \log(ed/s_0)}{|B|}\enspace.
\end{equation}
\end{proposition}

Proposition~\ref{prop:lasso_basis_estimators} is an (exact) oracle inequality with optimal residual term \cite{bellec2018slope}.
It is satisfied by the LASSO with a constant probability when trained on a set of informative data and for an optimal choice of regularization parameter $\lambda \sim \norm{\zeta}_{L_q}\sqrt{s_0 \log(ed/s_0)/N}$.
This regularization parameter requires the knowledge of the sparsity $s_0$. 
Proposition~\ref{prop:lasso_basis_estimators} shows that the risk bound only holds with constant probability because the noise is only assumed to have finite $L_q$-moment. 
Finally, the LASSO has to be trained with uncorrupted data; a single outlier completely breaks down its statistical properties---see Figure~1 in \cite{MOM2}. 

Let us now combine Corollary~\ref{Cor:AppliSelSousEns} and
Proposition~\ref{prop:lasso_basis_estimators} to apply the ensemble
method to this example.  Let $V,K_{\text{min}},K_{\text{max}}$ satisfy
the assumptions from Section~\ref{sec:Subsampling} and
Corollary~\ref{Cor:AppliSelSousEns}.  Denote by $s^*$ the largest 
integer $s$ such that $N/\max_{}(8V,2^{K_{\text{min}}+1})\geqslant
s\log(ed/s)$, and assume $1\leqslant \left\| \beta^* \right\|_0\leqslant s^*$ (where
$\left\|\,\cdot\,\right\|_0$ denotes the number of nonzero
coefficients).  Consider the set of subsamples $\mathcal{B}$ defined
as in Section~\ref{sec:defin-ensemble-meth}, the set
$\Lambda: =\left\{ c_0\left\| \zeta \right\| _{L_q}\sqrt{s\log (ed/s)} \right\}_{s\in [s^*]}$,
and $\mathcal{M}=\Lambda \times \mathcal{B}$. For $m=(\lambda,B)\in
\mathcal{M}$, consider the corresponding estimator $\hat{f}_m:=\hat{\beta}_{\lambda /\sqrt{\left| B \right| },B}$.
\begin{corollary}\label{coro:minmax_MOM_LASSO}
  Grant Assumptions~\ref{Ass:L4L2} and \ref{ass:lasso}.
  Let $\hat{m}$ be the output of the
ensemble method from Section~\ref{sec:defin-ensemble-meth}.
Then, with probability at least
$1-((s^*)^2N^2+1)\exp(-V/48)$, for all
$\varepsilon>0$,
\[  (1-a_{\varepsilon, V})
 \ell(\hat{\beta}_{\hat{\lambda},\hat{B}})
\leq (1+3 a_{\varepsilon, V})c_1 \norm{\zeta}_{L_q}^2\frac{\left\| \beta^* \right\|_0 \log(ed\left\| \beta^* \right\|_{0}^{-1})}{\left\lfloor N/\max_{}(4V,2^{K_{\text{min}}}) \right\rfloor } + 2 b_{\varepsilon, V}.  \]
\end{corollary}
While Proposition~\ref{prop:lasso_basis_estimators} shows statistical
guarantee with constant probability for the estimators $\hat
\beta_{\lambda, B}$ trained on uncorrupted data, Corollary~\ref{coro:minmax_MOM_LASSO}
shows that the ensemble method improves the constant probability into
an exponential probability, allows $|\cO|$ outliers as long as $V\geq
3 |\cO|$ and selects the best hyperparameter $\lambda$. 
The proof is given in Appendix~\ref{sec:proof-corollary-LASSO}.
A similar application to the ERM in linear aggregation is given in Appendix~\ref{sec:appl-erm-line}.

\section{Numerical experiments with the LASSO}
\label{sec:numer-exper}

\subsection{Presentation}
\label{sec:presentation}

In this section, the ensemble method from
Section~\ref{sec:Subsampling} is implemented and fed with the LASSO
algorithm, as in
Section~\ref{sec:applications}. Numerical
experiments 
 are performed with various amount and \emph{types} of outliers in order to investigate their effects on the output estimator $\hat f_{\hat m}$ and the corresponding parameter $(\lambda_{\hat m}, B_{\hat m})$.

We consider a framework with $2000$ features, i.e. $\mathbb{X}=\mathbb{R}^{2000}$ and let $\beta_0\in \mathbb{R}^{2000}$ which we assume $20$-sparse. The datasets are of size $N=1000$ and we consider the following numbers of outliers
$\left| \mathcal{O} \right| =0,4,8\dots, 150$. We construct two types
of outliers ($\mathcal{O}=\mathcal{O}_1 \sqcup \mathcal{O}_2$), both
of which are present in equal amount
($\left| \mathcal{O}_1\right|=\left| \mathcal{O}_2 \right|$).  The
first type, which we call \emph{hard outliers} are defined to simulate corruption due, for instance, to hardware issues:
$X_i=(1,\dots,1)\in\bR^{2000},\quad Y_i=10000,\quad i\in \mathcal{O}_1$, and
second type, which we call \emph{heavy-tail outliers} are constructed
as
$ X_i\sim \mathcal{N}(0,I_{2000}),\quad Y_i=\left< X_i , \beta_0
  \right> +\zeta_i,\quad i\in \mathcal{O}_2$, where the variables
$(X_i,Y_i)_{i\in \mathcal{O}_2}$ are i.i.d., $\zeta_i$ is a noise
independent of $X_i$ and distributed according to Student's
t-distribution with 2 degrees of freedom.
Informative data is drawn according to $X_i\sim \mathcal{N} (0,I_{2000}),\quad Y_i=\left< X_i , \beta_0
  \right> +\zeta_i,\quad i\in \mathcal{I}$,  where variables
$(X_i,Y_i)_{i\in \mathcal{I}}$ are i.i.d.,  and $\zeta_i$ ($i\in \mathcal{I}$) is a standard Gaussian
noise independent of $X_i$.
On the one hand, a \emph{hard outlier}, if contained in the training
sample of an estimator, is likely to significantly deteriorate its
performance. On the other hand, \emph{heavy-tail
  outliers} only differ from informative data in the distribution of
the noise, and should not deteriorate too much the performance of affected estimators. Nevertheless, we expect the informative data to be preferred over the type $2$ outliers in the selected subsample $B_{\hat m}$ (this is indeed the case in Figure~\ref{fig:nb-outliers-in-subsample}). 

We consider
$\Lambda=\left\{ e^k\,\middle|\,k\in \frac{1}{2}\left\llbracket -2,4
  \right\rrbracket \right\}$ as the grid of values for the
regularization parameter of the LASSO. We implement the ensemble
method from Section~\ref{sec:Subsampling} with parameters
$V=40$, $K_{\text{min}}=3$ and $K_{\text{max}}=4$.  The set
$\mathcal{B}$ of subsamples is constructed as in
Section~\ref{sec:an-effic-constr} and we set
$\mathcal{M}=\Lambda\times \mathcal{B}$. For each
$m=(\lambda,B)\in \mathcal{M}$, we train the LASSO estimator
$\hat{\beta}_{m}$ with hyperparameter $\lambda$ and subsample
$B$ (see \eqref{eq:lasso}).
We then compute the output estimator $\hat{\beta}_{\widehat{m}}$, which uses partitions
$(T_v^{(m,m')})_{v\in [V]}$ (for $m,m'\in \mathcal{M}$) constructed as
in Section~\ref{sec:an-effic-constr}.
Let us denote $\hat{\beta}_{\widetilde{m}}$ the best oracle estimator among
$(\hat{\beta}_m)_{m\in \mathcal{M}}$, in other words, let
$\tilde{m}:=\argmin_{m\in \mathcal{M}}R(\hat{\beta}_m)$ where
$\beta\mapsto R(\beta)=\norm{\beta-\beta_0}_2^2$ is the true risk
function which is not known---so that $\tilde{m}$ cannot be computed
using only the data. For comparison, we also compute the LASSO
estimators $\hat{\beta}_{\lambda,[N]}$  trained with the whole dataset, which we will call \emph{basic estimators},
  and let $\hat{\beta}_{\widetilde{\lambda},[N]}$ be the best among those, so that $\tilde{\lambda}:=\argmin_{\lambda\in \Lambda}R(\hat{\beta}_{\lambda,[N]})$. 

\subsection{On the choices of $V$ and $K_{\text{max}}$}
The choices of $V$ and $K_{\text{max}}$ have an impact on both the
performance the output estimator and the computation time.  The higher
is $V$, the higher is the number of outliers that the MOM-selection
procedure \eqref{def:MinmaxMOMSelect} can handle, and as a matter of
fact, Theorem~\ref{thm:oracle-inequality} requires
$V\geqslant 3\left| \mathcal{O} \right|$. However, higher values of
$V$ increase computation time and deteriorates the statistical
guarantee (through the values of $a_{\varepsilon,V}$ and
$b_{\varepsilon,V}$ from the statement of
Theorem~\ref{thm:oracle-inequality}). Here we choose $V=40$, so we can
expect the minimax-MOM selection procedure to perform well at least up until we
get as many as $\left\lfloor 40/3 \right\rfloor=13$ outliers.  The
number of considered subsamples is increasing with
$K_{\text{max}}$. High values of $K_{\text{max}}$ increase computation
time. Moreover, we don't want to go for the maximum value
$K_{\text{max}}=\left\lceil \log_2 N\right\rceil $, which would imply
the training of estimators with subsamples of size $2$, which is
irrelevant.  We therefore want a low value of $K_{\text{max}}$, but we
would like to have at least one subsample which contains no
outlier. This is necessarily the case when
$2^{K_{\text{max}}}>\left| \mathcal{O} \right|$.  Since the choice of
$V=40$ allows to hope for a good selection performance up to
$\left| \mathcal{O} \right|=13$ outliers, we choose $K_{\text{max}}=4$
which indeed satisfies $2^{K_{\text{max}}}>|\cO|$.

\subsection{Results and discussion}
\begin{figure}
  \centering
  \begin{subfigure}{.30\textwidth}
  \centering
    \includegraphics[width=\linewidth]{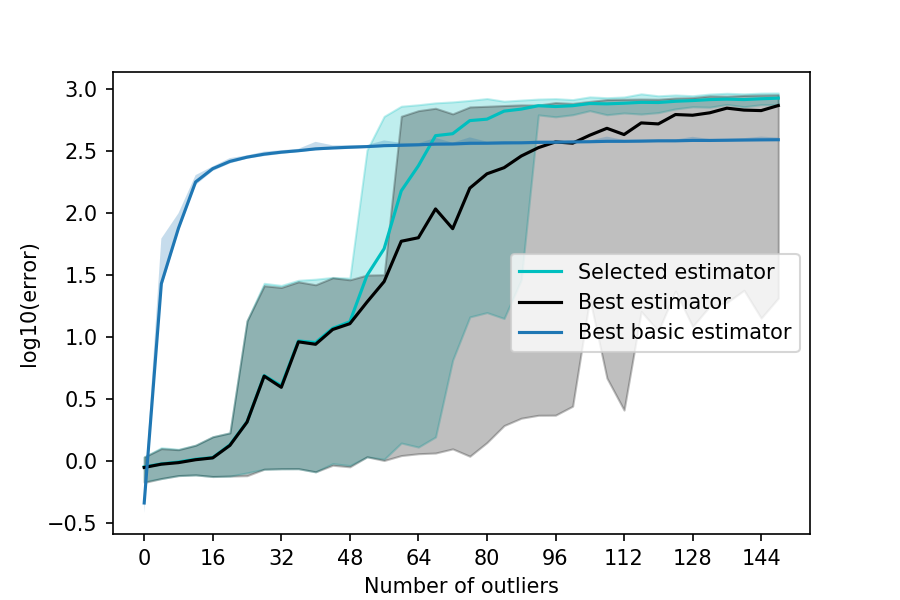}
\caption{Error of estimators $\hat{\beta}_{\widehat{m}}$,
  $\hat{\beta}_{\widetilde{m}}$ and $\hat{\beta}_{\widetilde{\lambda},[N]}$}
\label{fig:error}
  \end{subfigure}
  \begin{subfigure}{.30\textwidth}
  \centering
    \includegraphics[width=\linewidth]{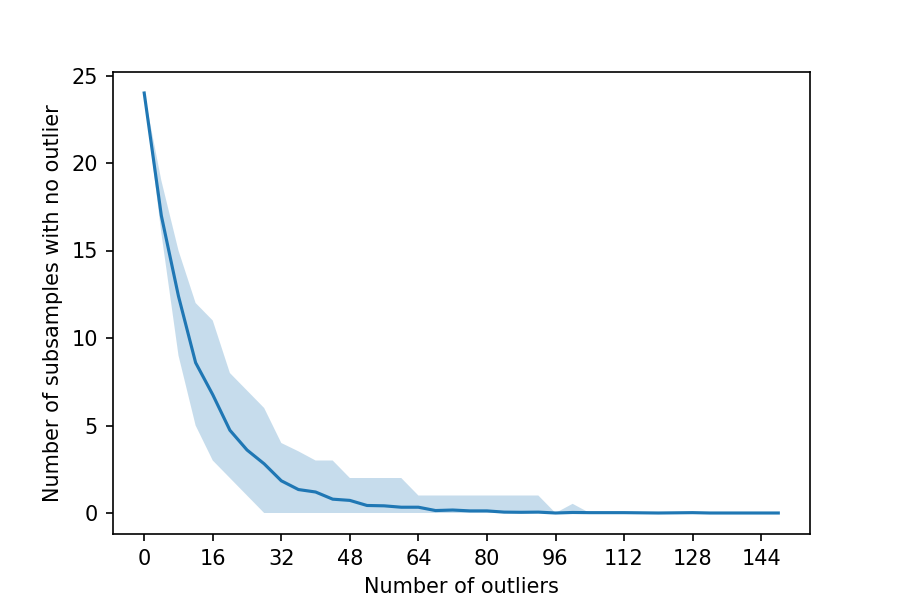}
    \caption{Number of subsamples with no outlier}
    \label{fig:subsamples-with-no-outlier}
  \end{subfigure}
  \begin{subfigure}{.30\textwidth}
  \centering
    \includegraphics[width=\linewidth]{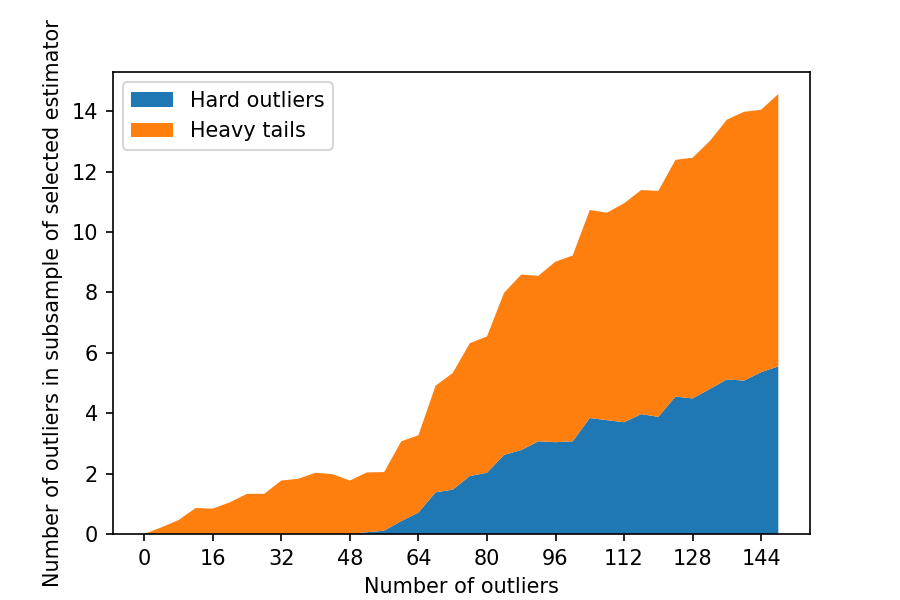}
    \caption{Number of outliers in $B_{\widehat{m}}$}
    \label{fig:nb-outliers-in-subsample}
  \end{subfigure}
  \caption{ensemble method run with $N=1000$, $V=40$,
    $K_{\text{max}}=4$, and averaged over 200 experiments.}
  \label{fig:figures}
\end{figure}

The plots presented in Figure~\ref{fig:figures} are averaged over $100$ experiments. Figure~\ref{fig:error} shows estimation error rates against the number
$\left| \mathcal{O} \right|$ of outliers in the dataset and a $95\%$
confidence interval 1) of the output estimator
$\hat{\beta}_{\widehat{m}}$ of the
ensemble method, 2) of $\hat{\beta}_{\widetilde{m}}$, the best estimator among $(\hat{\beta}_m)_{m\in \mathcal{M}}$, and 3) of the best
\emph{basic} estimator $\hat{\beta}_{\widetilde{\lambda},[N]}$.  
As soon as the dataset contains outliers, basic estimators
$(\hat{\beta}_{\lambda,[N]})_{\lambda\in \Lambda}$ have larger
errors than $\hat{\beta}_{\widetilde{m}}$ (the best estimator
computed on a subsample). 
For
$\left| \mathcal{O} \right| \leqslant 48$ the ensemble method 
procedure has the same error as the best estimator among
$(\hat{\beta}_m)_{m\in \mathcal{M}}$.  


For a given value of parameter $V$, the minimax-MOM selection procedure \eqref{def:MinmaxMOMSelect} is expected to fail at some point when the number of
outliers increases, but it seems here to resist to a much higher
number of outliers than predicted by the theory.
Theorem~\ref{thm:oracle-inequality} holds for
$\left| \mathcal{O} \right| \leqslant V/3$, that is
$\left| \mathcal{O} \right| \leqslant 13$ here.
It seems here that
the minimax-MOM selection procedure performs satisfactorily for
$\left| \mathcal{O} \right| \leqslant 48$ and even selects a
reasonably good estimator for $\left| \mathcal{O} \right| \leqslant
56$. 

Figure~\ref{fig:nb-outliers-in-subsample}
shows the number of each type of outliers in the selected subsample
$B_{\hat m}$.
The method manages to rule out hard outliers when
$\left| \mathcal{O} \right| \leqslant 48$, and the output estimator
$\hat\beta_{\hat m}$ has in these cases minimal risk, as the best
estimator $\hat\beta_{\widetilde{m}}$.
Figure~\ref{fig:subsamples-with-no-outlier} also shows that almost all
subsample contain outliers when $|\cO|\geq 48$.  Besides, the selected
subsample $B_{\hat m}$ contains heavy-tail outliers even for small
values of $\left| \mathcal{O} \right|$.  As heavy-tail outliers and
informative data define the same oracle, these heavy-tail outliers are
actually informative for the learning task and the minimax-MOM selection procedure
use this extra information automatically in an optimal way.  In
particular, the ensemble method distinguishes between non-informative
hard outliers and possibly informative heavy-tailed outliers.

Overall, our method shows very strong robustness to the presence of
outliers and outputs an estimator with the best possible performance
among the given class of estimators.




\small
\bibliography{biblio}
\bibliographystyle{plain}

\normalsize
\newpage
\appendix
\onecolumn

\section{Application to ERM and linear aggregation}
\label{sec:appl-erm-line}
This section applies Corollary~\ref{Cor:AppliSelSousEns} by
considering non-robust linear aggregation as input algorithms. 
%
Let $(F_{\lambda})_{\lambda\in \Lambda}$ be a finite collection of
subspaces of $F$, typically spanned by previous estimators. For each $\lambda\in \Lambda$,
denote by $d_{\lambda}$ the dimension of $F_{\lambda}$ and by
$f^*_{\lambda}$ an oracle in $F_{\lambda}$, meaning $f_{\lambda}^*:=\argmin_{f\in F_{\lambda}}R(f)$.
Denote $\hat{f}_{\lambda,B}$ the empirical risk minimizer (ERM) on
$F_{\lambda}$ trained with subsample $B$:
\begin{equation}\label{eq:ERM_agg_L}
\hat{f}_{\lambda,B}:=\argmin_{f\in F_{\lambda}}\frac{1}{\left| B \right| }\sum_{i\in B}^{}(Y_i-f(X_i))^2.
\end{equation}

The performance of ERM in linear aggregation like
$\hat{f}_{\lambda,B}$ under a $L_4/L_2$ assumption such as
Assumption~\ref{Ass:L4L2} have been obtained in \cite{MR3474824}.
\begin{proposition}[Theorem 1.3 in \cite{MR3474824}]\label{prop:GuillaumeShahar}
Let $\lambda\in \Lambda$. Assume that there exists $\chi_{\lambda}>0$ such that for all $f\in F_{\lambda}$, $(Pf^4)^{1/4}\leq \chi_\lambda (P f^2)^{1/2}$. Denote $\zeta_{\lambda}:=Y-f_{\lambda}^*(X)$ and assume that $(P\zeta_\lambda^4)^{1/4}\leq \sigma_\lambda$. Let $B\subset \mathcal{I}$ be such that $\left| B \right|\geqslant
(1600 \chi_\lambda^4)^2 d_{\lambda}$. Then, for every $x>0$, with probability larger than $1-\exp(-|B|/(64\chi_\lambda^8))-1/x$, the ERM $\hat{f}_{\lambda,B}$ defined in \eqref{eq:ERM_agg_L} satisfies  \[ \ell(\hat{f}_{\lambda,B})\leqslant \ell(f_{\lambda}^*)+(256)^2 \chi_\lambda^{12}\frac{\sigma_{\lambda}^2d_{\lambda}x}{\left| B \right| }.  \]
\end{proposition}
In Proposition~\ref{prop:GuillaumeShahar}, the (exact) oracle
inequality satisfied by $\hat{f}_{\lambda,B}$ guarantees an optimal
residual term of order $\sigma_\lambda^2 d_{\lambda}/N$ only when the
deviation parameter $x$ is constant. This may seem weak, but it cannot
be improved in general---see Proposition~1.5 in~\cite{MR3474824}: 
ERM are not robust to ``stochastic outliers'' in general. 

We can now combine these algorithms with our ensemble method.
Let $\mathcal{M}=\Lambda\times \mathcal{B}$ be as in
\eqref{eq:model_class} and $V$, $K_{\text{min}}$ and $K_{\text{max}}$
satisfy the assumptions from Section~\ref{sec:Subsampling} and
Corollary~\ref{Cor:AppliSelSousEns}. Consider the output estimator $\hat{f}_{\widehat{m}}$ from \eqref{def:MinmaxMOMSelect}. The following result combines Corollary~\ref{Cor:AppliSelSousEns} and Proposition~\ref{prop:GuillaumeShahar}.

\begin{corollary}\label{cor:ERM}
Grant Assumption~\ref{Ass:L4L2} on $F$ and assume that for all $\lambda\in\Lambda$ and all $f\in F_{\lambda}$, $(Pf^4)^{1/4}\leq \chi_\lambda (P f^2)^{1/2}$ and $(P\zeta_\lambda^4)^{1/4}\leq \sigma_\lambda$ for $\zeta_{\lambda}:=Y-f_{\lambda}^*(X)$. Assume also that  $N\geq \max_{\lambda\in\Lambda}(1600 \chi_\lambda^4)^2 d_{\lambda}\max_{}(8V,2^{K_{\text{min}}+1})$.
Then, with probability at least $1-(| \Lambda|^2 N^2+1)\exp(-V/48)$, for all $\varepsilon>0$,
\[    (1-a_{\varepsilon,V})\ell(\hat{f}_{\widehat{m}})\leqslant
   (1+3a_{\varepsilon,V})\times \min_{\lambda\in \Lambda}\left\{ \ell(f_{\lambda}^*)+2\exp(1/48)(256)^2\chi_\lambda^{12}\frac{\sigma_{\lambda}^2d_{\lambda}}{\lfloor N/\max_{}(4V,2^{K_{\text{min}}}) \rfloor} \right\}  +2b_{\varepsilon,V}.  \]
\end{corollary}
\begin{proof}
The proof follows from Proposition \ref{prop:GuillaumeShahar} and Corollary
\ref{Cor:AppliSelSousEns}. Let us check the assumption and the features of both results. For $x=2\exp(1/48)$ and when $|B|\geq (1600\chi_\lambda^4)^2d_\lambda$ we have $1-\exp(-|B|/(64\chi_\lambda^8))-1/x\geq 1-\exp(-1/48)$ therefore, $\hat f_{\lambda, B}$ satisfies an (exact) oracle inequality with probability larger than $1-\exp(-1/48)$ when $|B|\geq \nu(\lambda) := (1600 \chi_\lambda^4)^2d_\lambda $ with a residual term given by 
\begin{equation*}
\rho(\lambda, |B|) = \ell(f_{\lambda}^*)+2(256)^2\exp(1/48) \chi_\lambda^{12}\frac{\sigma_{\lambda}^2d_{\lambda}}{\left| B \right|}.
\end{equation*}Therefore, all the condition of Corollary \ref{Cor:AppliSelSousEns} are satisfied and the result follows from a direct application of the latter result.
\end{proof}

\section{Lemmas and proofs}\label{sec:Proof}
Let $\mathcal{V}^{(m,m')}:=\left\{ v\in
  [V]\,\middle|\,T_v^{(m,m')}\subset \mathcal{I} \right\}$ denote the
set of indices of blocks from the partition $(T_v^{(m,m')}:v\in[V])$ containing only informative data. In particular, we have
\begin{equation}
\label{eq:only-informative-card}
\left| \mathcal{V}^{(m,m')} \right| \geqslant V-\left| \mathcal{O} \right|.
\end{equation}

\begin{lemma}
\label{lm:conditional-variance}
Let $m,m'\in \mathcal{M}$ and $v\in \mathcal{V}^{(m,m')}$. The conditional variance of random variable
$P_{T_v^{(m,m')}}\left[ \gamma(\hat{f}_m)-\gamma(\hat{f}_{m'}) \right]$ given random variables $(X_i,Y_i)_{i\in B_m\cup B_{m'}}$ is bounded from above as:
\[ \operatorname{Var}\left( P_{T_v^{(m,m')}}\left[ \gamma(\hat{f}_m)-\gamma(\hat{f}_{m'}) \right]\,\middle|\, (X_i,Y_i)_{i\in B_m\cup B_{m'}} \right)\leqslant C_{m,m'},  \]
where
\[ C_{m,m'}:=\frac{V}{N}\left( 16\chi^4\left( \ell(\hat{f}_m)^2+\ell(\hat{f}_{m'})^2 \right)+64\sigma^2\left( \ell(\hat{f}_m)+\ell(\hat{f}_{m'}) \right) \right). \]
\end{lemma}
\begin{proof}
By assumption, random variables $(X_i,Y_i)_{i\in \mathcal{I}}$ are independent. In particular, random variables $(X_i,Y_i)_{i\in T_v^{(m,m')}}$ are independent conditionally to $(X_i,Y_i)_{i\in B_m\cup B_{m'}}$ since $v\in\mathcal{V}^{(m,m')}$. Using the shorthand notation $\operatorname{Var}_{m,m'}(\,\cdot\,):=\operatorname{Var}\left( \,\cdot\,\,\middle|\, (X_i,Y_i)_{i\in B_m\cup B_{m'}} \right)$, we have
\begin{align*}
\operatorname{Var}_{m,m'}\left( P_{T_v^{(m,m')}}\left[ \gamma(\hat{f}_m)-\gamma(\hat{f}_{m'}) \right]  \right)&=\operatorname{Var}_{m,m'}\left( \frac{1}{\left| T_v^{(m,m')} \right| }\sum_{i\in T_v^{(m,m')}}^{}(\gamma(\hat{f}_m)-\gamma(\hat{f}_{m'}))(X_i,Y_i) \right)\\  
&=\frac{1}{\left| T_v^{(m,m')} \right|^2 }\sum_{i\in T_v^{(m,m')}}^{}\operatorname{Var}_{m,m'}\left( (\gamma(\hat{f}_m)-\gamma(\hat{f}_{m'}))(X_i,Y_i) \right).
\end{align*}
Fix $i\in T_v^{(m,m')}\subset \mathcal{I}$, and let us bound from above each variance terms from the latter expression:
\begin{align*}
\operatorname{Var}_{m,m'}\left( (\gamma(\hat{f}_m)-\gamma(\hat{f}_{m'}))(X_i,Y_i) \right)&\leqslant \mathbb{E}\left[ \left(  (\gamma({\hat{f}_m})-\gamma(\hat{f}_{m'}))(X_i,Y_i) \right)^2\,\middle|\,(X_{i'},Y_{i'})_{i'\in B_m\cup B_{m'}} \right] \\ 
&=P\left[  \left( \gamma(\hat{f}_m)-\gamma(\hat{f}_{m'}) \right)^2 \right]\\
&= P\left[ (\gamma(\hat{f}_m)-\gamma(f^*)+\gamma(f^*)-\gamma(\hat{f}_{m'}))^2 \right]\\
&\leqslant 2\,  P\left[ (\gamma(\hat{f}_m)-\gamma(f^*))^2 \right]+2\, P\left[ (\gamma(\hat{f}_{m'})-\gamma(f^*))^2 \right],
\end{align*}
where we used the basic inequality $(x+y)^2\leqslant 2(x^2+y^2)$ in the last inequality.
Let us bound from above the first term. The second term is handled similarly. We use a quadratic/multiplier decomposition of the excess loss:
\begin{align*}
P\left[ (\gamma(\hat{f}_m)-\gamma(f^*))^2 \right] &=P\left[ \left( \left(  \hat{f}_m-f^* \right)^2-2(Y-f^*)(\hat{f}_m-f^*) \right)^2  \right]\\ 
&\leqslant 2 P\left[ (\hat{f}_m-f^*)^4 \right]+8 P\left[ (Y-f^*)^2(\hat{f}_m-f^*)^2 \right].
\end{align*}
By Assumption~\ref{Ass:L4L2}, it follows that
\begin{equation*}
P\left[ (\hat{f}_m-f^*)^4 \right]\leqslant \chi^4 \left( P\left[  (\hat{f}_m-f^*)^2 \right] \right)^2  = \chi^4\   \ell(\hat{f}_m)^2.
\end{equation*}
Likewise, Assumption~\ref{Ass:L4L2} yields
\begin{equation*}
P\left[ (Y-f^*)^2(\hat{f}_m-f^*)^2 \right]\leqslant \sigma^2 P\left[  (\hat{f}_m-f^*)^2 \right]  =\sigma^2  \ell(\hat{f}_m).
\end{equation*}
The result follows from combining these pieces and using $\left| T_v^{(m,m')} \right| \geqslant N/4V$.
\end{proof}

\begin{lemma}
\label{lm:encadrement-T}
With probability higher than $1-\left| \mathcal{M} \right|^2 e^{-(V-\left| \mathcal{O} \right| )/32}$, for all $m,m'\in \mathcal{M}$,
\[ \ell(\hat{f}_m)-\ell(\hat{f}_{m'})-\sqrt{8\, C_{m,m'}}\leqslant \mathcal{T}(m,m')\leqslant \ell(\hat{f}_m)-\ell(\hat{f}_{m'})+\sqrt{8\, C_{m,m'}}  \]
where $\mathcal{T}(m,m'):=\med_{v\in [V]}\left\{  P_{T_v^{(m,m')}}\left[ \gamma(\hat{f}_m)-\gamma(\hat{f}_{m'}) \right] \right\}$.
\end{lemma}
\begin{proof}
Fix $m,m'\in \mathcal{M}$ and $v\in \mathcal{V}^{(m,m')}$.
Conditionally to $(X_i,Y_i)_{i\in B_m\cup B_{m'}}$, it follows from Chebychev's inequality and Lemma \ref{lm:conditional-variance} that, with probability higher than $1-1/8$,
\begin{align*}
&\left| P_{T_v^{(m,m')}}\left[ \gamma(\hat{f}_m)-\gamma({\hat{f}_{m'}}) \right]- (\ell(\hat{f}_m)-\ell(\hat{f}_{m'})) \right|\\
 &\leqslant \sqrt{8\operatorname{Var}\left( P_{T_v^{(m,m')}}\left[ \gamma({\hat{f}_m})-\gamma({\hat{f}_{m'}}) \right]\,\middle|\, (X_i,Y_i)_{i\in B_m\cup B_{m'}} \right)}\\ 
&\leqslant \sqrt{8\, C_{m,m'}}.
\end{align*}
As the probability estimate does not depend on$(X_i,Y_i)_{i\in B_m\cup B_{m'}}$, the above also holds unconditionnally and,
with probability larger than $1-1/8$,
\begin{equation}
\label{eq:encadrement-P_T}
\ell(\hat{f}_m)-\ell(\hat{f}_{m'})-\sqrt{8\, C_{m,m'}} \leqslant P_{T_v^{(m,m')}}\left[ \gamma({\hat{f}_m})-\gamma({\hat{f}_{m'}}) \right] \leqslant \ell(\hat{f}_m)-\ell(\hat{f}_{m'})+\sqrt{8\, C_{m,m'}}.
\end{equation}

Denote by $\Omega_v^{(m,m')}$ the event defined by \eqref{eq:encadrement-P_T} and see that $\mathbb{P}\left[ \Omega_v^{(m,m')} \right]\geqslant 1-1/8$.  Apply now Hoeffding's inequality to random variables $\mathbbm{1}_{\Omega_v^{(m,m')}}, v\in \mathcal{V}^{(m,m')}$ which are independent conditionnally to $(X_i,Y_i)_{B_m\cup B_{m'}}$:
on an event $\Omega^{(m,m')}$ of probability larger than $1-e^{-2\left| \mathcal{V}^{(m,m')} \right| (1/8)^2}\geqslant 1-e^{-(V-\left| \mathcal{O} \right| )/32}$, see \eqref{eq:only-informative-card}, 
\begin{align*}
\frac{1}{\left| \mathcal{V}^{(m,m')} \right|}\sum_{v\in \mathcal{V}^{(m,m')}}^{}\mathbbm{1}_{\Omega_v^{(m,m')}}&\geqslant  \mathbb{E}\left[ \frac{1}{\left| \mathcal{V}^{(m,m')} \right| }\sum_{v\in \mathcal{V}^{(m,m')}}^{}\mathbbm{1}_{\Omega_v^{(m,m')}} \right] -\frac{1}{8}\\ 
&= \frac{1}{\left| \mathcal{V}^{(m,m')} \right| }\sum_{v\in \mathcal{V}^{(m,m')}}^{}\mathbb{P}\left[ \Omega_v^{(m,m')} \right]-\frac{1}{8}\geqslant\frac{3}{4}.
\end{align*}
Then, on $\Omega^{(m,m')}$, using \eqref{eq:only-informative-card} and the assumption $V\geq 3|\cO|$,
\[ \sum_{v\in \mathcal{V}^{(m,m')}}^{}\mathbbm{1}_{\Omega_v^{(m,m')}}\geqslant \frac{3}{4}\left| \mathcal{V}^{(m,m')} \right|\geqslant \frac{3}{4}\left( V-\left| \mathcal{O} \right|  \right)  \geqslant \frac{V}{2}.\]
 In other words, inequalities (\ref{eq:encadrement-P_T}) hold for more
than half of the indices $v\in [V]$. Therefore, on
event $\Omega^{(m,m')}$, the same inequality holds for the median over $v\in [V]$:
\begin{equation*}
\ell(\hat{f}_m)-\ell(\hat{f}_{m'})-\sqrt{8 C_{m,m'}}\leqslant \mathcal{T}(m,m')\leqslant \ell(\hat{f}_m)-\ell(\hat{f}_{m'})+\sqrt{8\, C_{m,m'}}.  
\end{equation*}
By a union bound, the above holds for all $m,m'\in \mathcal{M}$ with probability at least $1-\left| \mathcal{M} \right|^2 e^{-(V-\left| \mathcal{O} \right| )/32}$.
\end{proof}


\begin{lemma}
\label{lm:C}
For all $m,m'\in \mathcal{M}$, $\varepsilon^\prime>0$ and $b>0$,
\[ \sqrt{8\, C_{m,m'}}\leqslant\sqrt{\frac{8V}{N}}\left( (4\chi^2+\varepsilon^\prime)(\ell(\hat{f}_m)+\ell(\hat{f}_{m'}))+\frac{16\sigma^2}{\varepsilon^\prime} \right).  \]
\end{lemma}
\begin{proof}
By definition of $C_{m,m'}$ and the inequalities $\sqrt{x+y}\leqslant
\sqrt{x}+\sqrt{y}$ and $2xy\leqslant
x^2/\varepsilon^\prime+\varepsilon^\prime y^2$,
\begin{align*}
\sqrt{8 C_{m,m'}}&=\sqrt{\frac{8V}{N}\left( 16\chi^4\left( \ell(\hat{f}_m)^2+\ell(\hat{f}_{m'})^2 \right)+64\sigma^2\left( \ell(\hat{f}_m)+\ell(\hat{f}_{m'}) \right)   \right) }\\
&\leqslant \sqrt{\frac{8V}{N}}\left( 4\chi^2\sqrt{\ell(\hat{f}_m)^2+\ell(\hat{f}_{m'})^2}+8\sigma\sqrt{\ell(\hat{f}_m)+\ell(\hat{f}_{m'})} \right),\\ 
&\leqslant \sqrt{\frac{8V}{N}}\left( 4\chi^2(\ell(\hat{f}_m)+\ell(\hat{f}_{m'}))+\frac{16\sigma^2}{\varepsilon^\prime}+\varepsilon^\prime(\ell(\hat{f}_m)+\ell(\hat{f}_{m'})) \right)\\
&=\sqrt{\frac{8V}{N}}\left( (4\chi^2+\varepsilon^\prime)(\ell(\hat{f}_m)+\ell(\hat{f}_{m'}))+\frac{16\sigma^2}{\varepsilon^\prime} \right).
\end{align*}
\end{proof}

\begin{lemma}
\label{lm:encadrement-T-2}
With probability at least $1-\left| \mathcal{M} \right|^2e^{-(V-\left| \mathcal{O} \right| )/32}$, for all $m,m'\in \mathcal{M}$ and $\varepsilon>0$:
\[ (1-a_{\varepsilon,V})\ell(\hat{f}_m)-(1+a_{\varepsilon,V})\ell(\hat{f}_{m'})-b_{\varepsilon,V} \leqslant \mathcal{T}(m,m') \leqslant (1+a_{\varepsilon,V})\ell(\hat{f}_m)-(1-a_{\varepsilon,V})\ell(\hat{f}_{m'})+b_{\varepsilon,V}.  \]
\end{lemma}
\begin{proof}
The result follows from Lemmas \ref{lm:encadrement-T} and \ref{lm:C} for $\varepsilon^\prime = \sqrt{N/V}\varepsilon$, together
with the definition of $a_{\varepsilon,V}$ and $b_{\varepsilon,V}$.
\end{proof}

\section{Proofs of the main results}\label{Proofs:Main}
\subsection{Proof of Theorem~\ref{thm:oracle-inequality}}\label{sec:ProofMainThm}
Assume that $a_{\varepsilon,V}<1$, otherwise the result is trivial.
Denote $m_o:=\argmin_{m\in \mathcal{M}}\ell(\hat{f}_m)$, so
\begin{equation}\label{eq:first_ineq_coro_main}
(1-a_{\varepsilon,V})\ell(\hat{f}_{\widehat{m}}) = (1-a_{\varepsilon,V})  \ell(\hat{f}_{\widehat{m}})- (1+a_{\varepsilon,V})\ell(\hat{f}_{m_o})+(1+a_{\varepsilon,V})\ell(\hat{f}_{m_o}) . 
\end{equation}

Let $\Omega$ be the event defined by Lemma \ref{lm:encadrement-T-2}. Since $V\geqslant 3|\cO|$, by Lemma \ref{lm:encadrement-T-2}
\[
\bP\pa{\Omega}\geqslant 1-\left| \mathcal{M} \right|^2e^{-(V-\left| \mathcal{O} \right| )/32}\geqslant 1-\left| \mathcal{M} \right|^2e^{-V/48}\enspace.
\]
It follows from Lemma~\ref{lm:encadrement-T-2} and \eqref{eq:first_ineq_coro_main} that, on $\Omega$,
\begin{equation}
\label{eq:3}
(1-a_{\varepsilon,V})\ell(\hat{f}_{\widehat{m}})\leqslant \max_{m\in \mathcal{M}}\mathcal{T}(\hat{m},m)+b_{\varepsilon,V}+(1+a_{\varepsilon,V})\ell(\hat{f}_{m_o}).
\end{equation}
Then, by definition of $\hat{m}$ and using Lemma \ref{lm:encadrement-T-2}, on $\Omega$,
\begin{align*}
\max_{m\in \mathcal{M}}\mathcal{T}(\hat{m},m)&=\min_{m'\in \mathcal{M}}\max_{m\in \mathcal{M}}\mathcal{T}(m',m)\leqslant \max_{m\in \mathcal{M}}\mathcal{T}(m_o,m)\\
&\leqslant \max_{m\in \mathcal{M}}\left\{ (1+a_{\varepsilon,V})\ell(\hat{f}_{m_o})-(1-a_{\varepsilon,V})\ell(\hat{f}_m)+b_{\varepsilon,V} \right\}\\ 
&=(1+a_{\varepsilon,V})\ell(\hat{f}_{m_o})-(1-a_{\varepsilon,V})\ell(\hat{f}_{m_o})+b_{\varepsilon,V}=2a_{\varepsilon,V}\ell(\hat{f}_{m_o})+b_{\varepsilon,V},
\end{align*}
where we used $1-a_{\varepsilon,V}\geqslant 0$ and the definition of $m_o$. Plugging this into (\ref{eq:3}) yields the result.

\subsection{Proof of Corollary~\ref{Cor:AppliSelSousEns}}\label{sec:ProofCorSubSel}
Let
\[ K_0:=\max_{}\left( \left\lceil  \log_2(2V) \right\rceil,K_{\text{min}} \right). \]
It follows from the assumption $V\leqslant 2^{K_{\text{max}}-1}$
that $K_0\in \left\llbracket K_{\text{min}},K_{\text{max}} \right\rrbracket$.
Besides, it follows from the above definition that:
\begin{equation}
\label{eq:1}
2^{K_0}\leqslant \max_{}(4V,2^{K_{\text{min}}}).
\end{equation}
Let also
\begin{gather*}
 \lambda_0:=\argmin_{\lambda\in \Lambda}\rho(\lambda,\lfloor N/2^{K_0} \rfloor)\qquad
    \text{and}\qquad \rho_0:=\rho(\lambda_0,\lfloor N/2^{K_0} \rfloor  ).\\
  \mathcal{K}_0:=\left\{ k\in [2^{K_0}]\,\middle|\,B_k^{(K_0)}\subset \mathcal{I} \right\},    
\end{gather*}
   which is nonempty because $\left| \mathcal{K}_0 \right| \geqslant
  2^{K_0}-\left| \mathcal{O} \right|  \geqslant 2V-\left| \mathcal{O}
  \right| \geqslant 2V-V/3\geqslant V$, where the first inequality
 follows from the definition of $\mathcal{K}_0$, the second inequality
 from the definition of $K_0$ and the third inequality from the
 assumption $V\geqslant 3\left| \mathcal{O} \right|$.

 Consider the events
\begin{align*}
\Omega_1&=\left\{
          (1-a_{\varepsilon,V})\ell(\hat{f}_{\widehat{m}})\leqslant
          (1+3a_{\varepsilon,V})\min_{m\in
          \mathcal{M}}\ell(\hat{f}_m)+2b_{\varepsilon,V} \right\}\\
\Omega_2&=\left\{ \exists k\in \mathcal{K}_0,\ \ell\left( \hat{f}_{\lambda_0,B_k^{(K_0)}} \right)\leqslant \rho_0  \right\}.  
\end{align*}
   From now on, assume that $\Omega_1\cap \Omega_2$ hold and the aim
   is to establish an upper bound on $\min_{m\in
     \mathcal{M}}\ell(\hat{f}_m)$.
 Write
\begin{align*}
\min_{m\in \mathcal{M}}\ell(\hat{f}_m)=\min_{\substack{\lambda\in
                                        \Lambda\\B\in
  \mathcal{B}}}\ell(\hat{f}_{\lambda,B})\leqslant \min_{k\in
  \mathcal{K}_0} \ell\left( \hat{f}_{\lambda_0,B_k^{(K_0)}}
  \right)&\leqslant \rho_0 =\min_{\lambda\in
  \Lambda}\rho(\lambda,\lfloor N/2^{K_0} \rfloor  )\\
&\leqslant \min_{\lambda\in \Lambda}\rho\left( \lambda, \left\lfloor \frac{N}{\max_{}(4V,2^{K_{\text{min}}})} \right\rfloor \right).
\end{align*}
Here the second inequality comes from the definition of
$\Omega_2$ and the last inequality from \eqref{eq:1}
combined with $\rho$ being nonincreasing in its second variable. Combining the above with the definition of $\Omega_1$ yields the desired inequality:
\[ (1-a_{\varepsilon,V})\ell(\hat{f}_{\widehat{m}})\leqslant
  (1+3a_{\varepsilon,V})\, \min_{\lambda\in \Lambda}\rho\left(
    \lambda, \left\lfloor  \frac{N}{\max_{}(4V,2^{K_{\text{min}}})} \right\rfloor\right)+2b_{\varepsilon,V}. \]

To conclude the proof, let us bound from below the probability of $\Omega_1\cap \Omega_2$. By Theorem \ref{thm:oracle-inequality},
\begin{align*}
\mathbb{P}\left[ \Omega_1\cap \Omega_2 \right]&=1-\mathbb{P}\left[ \Omega_1^c\cup \Omega_2^c \right]\geqslant 1-\mathbb{P}\left[ \Omega_1^c \right]-\mathbb{P}\left[ \Omega_2^c \right]\\  
&\geqslant 1-\left| \mathcal{M} \right|^2e^{-V/48}-\mathbb{P}\left[ \Omega_2^c \right]\geqslant 1-\left| \Lambda \right|^2N^2e^{-V/48}-\mathbb{P}\left[ \Omega_2^c \right],
\end{align*}
Recall that $\lfloor N/2^{K_0} \rfloor\leqslant \left|
  B_k^{(K_0)} \right|$ for all $k\in \left[ 2^{K_0} \right]$, that $\rho$ is non-increasing in its second variable and that $B_{k}^{(K_0)}, k\in \cK_0$ are disjoint, so
\begin{align*}
\mathbb{P}\left[ \Omega_2^c \right] &=\mathbb{P}\left[ \forall k\in \mathcal{K}_0,\ \ell\left(\hat{f}_{\lambda_0,B_k^{(K_0)}} \right)>\rho_0 \right]= \prod_{k\in \mathcal{K}_0}\mathbb{P}\left[ \ell\left(\hat{f}_{\lambda_0,B_k^{(K_0)}} \right)>\rho_0 \right]\\ 
&\leqslant \prod_{k\in \mathcal{K}_0}\mathbb{P}\left[ \ell\left(\hat{f}_{\lambda_0,B_k^{(K_0)}} \right)>\rho \left(\lambda_0,\left| B_k^{(K_0)} \right| \right) \right]\\ &\leqslant \exp(-|\cK_0|/48)\leq \exp(-V/48)\enspace.
\end{align*}
The third inequality follows from the excess risk bound on
$\hat{f}_{\lambda_0,B_k^{(K_0)}}$, which holds as soon as
$\left| B_k^{(K_0)} \right| \geqslant \nu(\lambda_0)$; this is indeed
case because using \eqref{eq:1}:
\begin{align*}
\left| B_k^{(K_0)} \right| \geqslant \left\lfloor \frac{N}{2^{K_0} }\right\rfloor &\geqslant
                                                                                    \left\lfloor \frac{N}{\max_{}(4V,2^{K_{\text{min}}})} \right\rfloor=\frac{N}{2\max_{}(2V,2^{K_{\text{min}}})}
 \geqslant \nu_{\text{max}}\geqslant \nu(\lambda_0),
\end{align*}
where the penultimate inequality follows from assumption $N\geqslant \nu_{\text{max}}\max_{}(8V,2^{K_{\text{min}}+1})$.
Therefore, $\mathbb{P}\left[ \Omega_1\cap \Omega_2 \right]\geqslant 1-(\left| \Lambda \right|^2N^2 + 1)\exp(-V/48)$. 
\subsection{Proof of Lemma~\ref{lem:key1}}
Start with \textit{(i)}. Let $k\in [2^K]$ and $i\in B_k^{(K)}$, which by definition of $B_k^{(K)}$ means:
  \[ \left\lfloor \frac{(k-1)N}{2^K} \right\rfloor< i\leqslant
    \left\lfloor \frac{kN}{2^K} \right\rfloor. \]
  We can bound from below as follows: let $k':=\lfloor(k-1)(2^{K'-K})\rfloor+1$,
  \[ \left\lfloor \frac{(k-1)N}{2^K} \right\rfloor = \left\lfloor
      \frac{(k-1)(2^{K'-K})N}{2^{K'}} \right\rfloor\geqslant
    \left\lfloor \frac{\lfloor(k-1)(2^{K'-K})\rfloor N}{2^K}
    \right\rfloor = \left\lfloor \frac{(k'-1)N}{2^K} \right\rfloor. \]
 Similarly, the
  upper bound is obtained as follows:
\begin{align*}
\left\lfloor \frac{kN}{2^K} \right\rfloor &= \left\lfloor
                                            \frac{((k-1)2^{K'-K}+2^{K'-K})N}{2^{K'}}
                                            \right\rfloor\leqslant
                                            \left\lfloor
                                            \frac{((k-1)2^{K'-K}+2)N}{2^{K'}}
                                            \right\rfloor\\
&\leqslant \left\lfloor \frac{(\lfloor(k-1)2^{K'-K}\rfloor+1)N}{2^{K'}} \right\rfloor=\left\lfloor \frac{k'N}{2^{K'}} \right\rfloor.
\end{align*}
  Therefore,
  \[ \left\lfloor \frac{(k'-1)N}{2^{K'}} \right\rfloor<i\leqslant \left\lfloor \frac{k'N}{2^{K'}} \right\rfloor. \]
This means $i\in B_{\lfloor(k-1)2^{K'-K}\rfloor +1}^{(K')}$.

The proof of \textit{(ii)} proceeds as follows. Let $k'\in [2^{K'}]$ and $i\in B_{k'}^{(K')}$, meaning that
  \[ \left\lfloor \frac{(k'-1)N}{2^{K'}} \right\rfloor<i\leqslant \left\lfloor \frac{k'N}{2^{K'}} \right\rfloor. \]
This can be rewritten as
  \[ \left\lfloor \frac{2^{K-K'}(k'-1)N}{2^{K}} \right\rfloor<i\leqslant \left\lfloor \frac{2^{K-K'}k'N}{2^{K}} \right\rfloor, \]
  As $(B_k^{(K)})_{k\in [2^K]}$ is a partition of $[N]$, there exists a
  unique $k\in [2^K]$ such that $i\in B_k^{(K)}$, i.e.
  \[ \left\lfloor \frac{(k-1)N}{2^K} \right\rfloor< i\leqslant
    \left\lfloor \frac{kN}{2^K} \right\rfloor. \]
  Combining the two previous displays implies that $(k'-1)2^{K-K'}<k\leqslant k'2^{K-K'}$,  meaning that $i\in B_k^{(K)}$. Moreover, applying (\ref{item:2}) to $k$ gives that $B_k^{(K)}$ is a
  subset of $B_{k'}^{(K')}$, hence the result.
\subsection{Proof of Lemma~\ref{lm:1}}
  Using (\ref{item:2}) from  Lemma~\ref{lem:key1}, we have:
\begin{align*}
B_{k_1}^{K_1}&\subset B_{k_1'}^{(3)},\quad \text{with}\quad  k_1'=\lfloor(k_1-1)2^{3-K_1}\rfloor +1,\\
B_{k_2}^{K_2}&\subset B_{k_2'}^{(3)},\quad \text{with} \quad k_2'=\lfloor(k_2-1)2^{3-K_2}\rfloor +1.
\end{align*}
In addition, $K_0$ is by definition larger than $3$ (see
\eqref{eq:2}) and (\ref{item:3}) from Lemma~\ref{lm:1} then gives:
\begin{align*}
(B_{k}^{(K_0)})_{k\in \left\rrbracket   (k_1'-1)2^{K_0-3},\ k_1'2^{K_0-3} \right\rrbracket}&\quad \text{is a
                                                     partition of}\quad B_{k_1'}^{(3)},\\
(B_{k}^{(K_0)})_{k\in \left\rrbracket   (k_2'-1)2^{K_0-3},\ k_2'2^{K_0-3} \right\rrbracket}&\quad \text{ is a
                                                     partition of
                                                     }\quad B_{k_2'}^{(3)}.
\end{align*}
Then, using the latter result and starting from the definition of $\mathcal{K}_0(K_1,k_1,K_2,k_2)$, we can write
\begin{align*}
\mathcal{K}_0(K_1,k_1,K_2,k_2)&=\left\{ k\in \left[ 2^{K_0}
    \right]\,\middle|\,B_k^{(K_0)}\cap (B_{k_1}^{(K_1)}\cup
  B_{k_2}^{(K_2)})=\varnothing \right\}\\
&\supset \left\{ k\in \left[ 2^{K_0}
                                            \right]\,\middle|\,B_k^{(K_0)}\cap (B_{k_1'}^{(3)}\cup B_{k_2'}^{(3)})=\varnothing \right\}\\
&= \left[ 2^{K_0} \right]\setminus \left( \left\rrbracket   (k_1'-1)2^{K_0-3},\ k_1'2^{K_0-3} \right\rrbracket \cup \left\rrbracket   (k_2'-1)2^{K_0-3},\ k_2'2^{K_0-3} \right\rrbracket  \right).
\end{align*}
The latter result together with the definition of $K_0$ yield the lower bound on the cardinality of $\mathcal{K}_0(k_1,K_1,K_2,k_2)$:
\begin{align*}
\left| \mathcal{K}_0(K_1,k_1,K_2,k_2) \right| &\geqslant
                                                2^{K_0}-2^{K_0-3}-2^{K_0-3}=\frac{3}{4}2^{K_0}\geqslant
                                                \frac{3}{4}2^{\log_2(V/3)+2}=V.
\end{align*}
\subsection{Proof of Lemma~\ref{lem:key3}}
Let $(m,m')\in \mathcal{M}^2$ and $v\in [V]$. By construction of
$T_v^{(m,m')}$, there exists $k\in \left[ 2^{K_0} \right]$ such that
$T_v^{(m,m')}=B_k^{(K_0)}$. Then, it follows from \eqref{eq:2} that $2^{K_0}\leqslant 8V/3$, so that we can write:
\[ \left| T_v^{(m,m')} \right| \geqslant \left\lfloor
    \frac{N}{2^{K_0}} \right\rfloor\geqslant \left\lfloor \frac{3N}{8V}
  \right\rfloor= \left\lfloor \frac{N}{4V}+\frac{N}{8V}
  \right\rfloor \geqslant \left\lfloor \frac{N}{4V}+1 \right\rfloor \geqslant \frac{N}{4V}, \]
where we used the fact that $V\leqslant N/8$ by assumption. 

\subsection{Proof of Corollary~\ref{coro:minmax_MOM_LASSO}}
\label{sec:proof-corollary-LASSO}
Let us apply Corollary~\ref{Cor:AppliSelSousEns}
with well-chosen functions $\nu\colon \lambda\to \mathbb{R}_+^*$ and
$\rho\colon \Lambda\times \mathbb{N}^*\to \mathbb{R}_+\cup \left\{
  +\infty \right\}$. For $s\in [s^*]$, denote,
\[ \lambda_s=c_0\left\| \zeta \right\|_{L_q}\sqrt{s\log (ed/s)} \]
so that $\Lambda=\left\{ \lambda_s \right\}_{s\in [s^*]}$.  For $s\in [s^*]$,
we define
\[ \nu(\lambda_s)=s\log (ed/s). \]  Then, by definition of
$s^*$, we get
$\nu_{\text{max}}=\left\lceil \max_{\lambda\in \Lambda}\nu(\lambda)
\right\rceil=\left\lceil s^*\log (ed/s^*) \right\rceil$, and
the assumption $N/(8V,2^{K_{\text{min}}+1})\geqslant s_*\log (ed/s^*)$
implies $N\geqslant \nu_{\text{max}}\max_{}(8V,2^{K_{\text{min}}+1})$
which is required to apply Corollary~\ref{Cor:AppliSelSousEns}.
Besides, we define $\rho$ as
follows. For $s\in [s^*]$ and $c\in \mathbb{N}^*$,

\[
  \rho(\lambda_s,c)=\begin{cases}
  \lambda_s^2c_1c_0^{-2}c^{-1}  &\text{if
$s\geqslant \left\| \beta^* \right\|_0$}\\
+\infty&\text{otherwise}.
  \end{cases}
\]
For all $s\geqslant \left\| \beta^* \right\|_0$, $\beta^*$ is of
course $s$-sparse, and then according to
Proposition~\ref{prop:lasso_basis_estimators}, for $B\subset
\mathcal{I}$ such that $\left| B \right|\geqslant \nu(\lambda_s)$, it
holds with probability higher than $1-e^{-1/48}$ that
\[ \ell(\hat{f}_{\lambda_s,B})=\ell(\hat{\beta}_{\lambda_s\left| B \right|^{-1/2},B})\leqslant \rho(\lambda_s,\left| B \right| ). \]
The above inequality is also true for $s<\left\| \beta^* \right\|_0$
since the right-hand side is then infinite. We can now apply
Corollary~\ref{Cor:AppliSelSousEns} which gives that with probability
higher than $1-((s^*)^2N^2+1)e^{-V/48}$, it holds that:
\[ 
 (1-a_{\varepsilon,V})\ell(\hat{f}_{\widehat{m}})\leqslant
  (1+3a_{\varepsilon,V})\, \min_{\lambda\in \Lambda}\rho\left(   \lambda,\left\lfloor  \frac{N}{\max_{}(4V,2^{K_{\text{min}}})} \right\rfloor\right)+2b_{\varepsilon,V},
 \]
and the result follows by noting that:
\begin{align*}
\min_{\lambda\in \Lambda}\rho\left(   \lambda,\left\lfloor
  \frac{N}{\max_{}(4V,2^{K_{\text{min}}})}
  \right\rfloor\right)&=\min_{s\in [s^*]}\rho\left( \lambda_s,\
  \left\lfloor  \frac{N}{\max_{}(4V,2^{K_{\text{min}}})} \right\rfloor
  \right)\\
&=\min_{s\geqslant \left\| \beta^* \right\|_0}\lambda_s^2c_1c_0^{-2}\left\lfloor
              \frac{N}{\max_{}(4V,2^{K_{\text{min}}})}
              \right\rfloor^{-1}\\
&=\min_{s\geqslant \left\| \beta^* \right\|_0}\left\| \zeta \right\|_{L_q}^2s\log (ed/s)c_1\left\lfloor
              \frac{N}{\max_{}(4V,2^{K_{\text{min}}})}
              \right\rfloor^{-1}\\
&=c_1\left\| \zeta \right\| _{L_q}^2\frac{\left\| \beta^* \right\|_0\log (ed\left\| \beta^* \right\|_0^{-1})}{\left\lfloor N/\max_{}(4V,2^{K_{\text{min}}}) \right\rfloor  }.
\end{align*}

\end{document}

%% file: commandes_guillaume.tex
\newcommand{\inr}[1]{\bigl< #1 \bigr>}

\newcommand{\norm}[1]{\left\|#1\right\|}%

\newcommand{\pa}[1]{\left(#1\right)}

\def\ds1{\textrm{1\kern-0.25emI}} 

\newcommand \E{\mathbb{E}}
\newcommand \R{\mathbb{R}}

\newcommand \cB{{\cal B}}

\newcommand \cI{{\cal I}}

\newcommand \cK{{\cal K}}

\newcommand \cM{{\cal M}}

\newcommand \cO{{\cal O}}

\newcommand \bP{{\mathbb P}}

\newcommand \bR{{\mathbb R}}

